\documentclass[12pt]{article}

\usepackage{amsfonts}
\usepackage{amsmath}
\usepackage{amsthm}
\newtheorem{theorem}{Theorem}

\newtheorem{q}[theorem]{Question}
\newtheorem{conj}[theorem]{Conjecture}
\newtheorem{cor}[theorem]{Corollary}

\def\F{\mathcal{F}}
\def\x{{\bf x}}
\def\y{{\bf y}}

\begin{document}

\title{Matchings and Hamilton Cycles with Constraints on Sets of Edges}
\author{J.~Robert Johnson\thanks{School of Mathematical Sciences, Queen Mary University of London, London E1 4NS (\tt{r.johnson@qmul.ac.uk})}} 

\maketitle

\begin{abstract}
The aim of this paper is to extend and generalise some work of Katona on the existence of perfect matchings or Hamilton cycles in graphs subject to certain constraints. The most general form of these constraints is that we are given a family of sets of edges of our graph and are not allowed to use all the edges of any member of this family. We consider two natural ways of expressing constraints of this kind using graphs and using set systems.

For the first version we ask for conditions on regular bipartite graphs $G$ and $H$ for there to exist a perfect matching in $G$, no two edges of which form a $4$-cycle with two edges of $H$.

In the second, we ask for conditions under which a Hamilton cycle in the complete graph (or equivalently a cyclic permutation) exists, with the property that it has no collection of intervals of prescribed lengths whose union is an element of a given family of sets. For instance we prove that the smallest family of $4$-sets with the property that every cyclic permutation of an $n$-set contains two adjacent pairs of points has size between $(1/9+o(1))n^2$ and $(1/2-o(1))n^2$.
We also give bounds on the general version of this problem and on other natural special cases.

We finish by raising numerous open problems and directions for further study. 

\end{abstract}

\section{Introduction}

Many results in graph theory concern establishing conditions on a graph $G$ which guarantee that $G$ must contain some particular spanning structure. A classical example of this is Dirac's theorem \cite{Dirac}:

\begin{theorem}[Dirac]\label{dirac:thm}
Every graph on $n$ vertices with minimum degree at least $n/2$ has a Hamilton cycle.
\end{theorem}

Demetrovics, Katona, and Sali \cite{DKS} proved an extension of this in which there is a second graph $H$ on the same vertex set as $G$, and we are required to find a Hamilton cycle in $G$ which satisfies the condition that no two edges of it form an ``alternating cycle'' with two edges of $H$. More precisely they proved:

\begin{theorem}[Demetrovics, Katona, and Sali]\label{dks:thm}
Let $G$ and $H$ be graphs with $V(G)=V(H)$ and $E(G)\cap E(H)=\emptyset$. 
Let $|V(G)|=n$, $r$ be the minimum degree of $G$, and $s$ be the maximum degree of $H$. Provided that
\[
2r-8s^2-s-1>n
\]
there is a Hamilton cycle in $G$ such that if $(a,b)$ and $(c,d)$ are both edges
of the cycle, then $(b,c),(d,a)$  are not both edges of $H$.
\end{theorem}

We mention briefly that this theorem was proved with an application involving pairing up sets in mind. In this application, the graph $G$ has as vertices the $r$-element subsets of $[m]$ with disjoint subsets being adjacent. By a suitable choice of $H$ it was shown that the $r$-element subsets of $[m]$ can be paired up into disjoint pairs with distinct pairs being significantly different in some natural quantitative sense. See \cite{DKS} and \cite{K} for more details.

Some further results on finding Hamilton cycles in graphs under constraints of this kind were proved and some questions were raised in \cite{K}. The aim of this paper is to answer, completely or partially, some of these questions, and to give some further results, questions and conjectures which may lead towards a more general theory of problems of this kind.

In Section 2 we consider the following question and related problems. 

\begin{q}\label{regularmatch:q}
Let $A,B$ be fixed $n$-sets and let $G,H$ be bipartite graphs on $A\cup B$ with bipartition $(A,B)$. Further, suppose that $G$ is $r$-regular, $H$ is $s$-regular and $G$ and $H$ have disjoint edge sets. For which $n,r,s$ are we guaranteed that there exists a perfect matching in $G$ in which no two edges of the matching form a 4-cycle with two edges of $H$.
\end{q}

Answers to this can be thought of as bipartite analogues of Theorem \ref{dks:thm}. The replacement of the Hamilton cycle in a graph of Theorem \ref{dks:thm} with a perfect matching in a bipartite graph is natural particularly given the pairing up sets applications from \cite{DKS} and \cite{K}. We give conditions on $r$ and $s$ for a perfect matching of this form to exist. In the case that $G$ and $H$ are bipartite complements (in the sense that they are edge disjoint and their union is $K_{n,n}$) we are able to bound the smallest $s$ for which a suitable perfect matching is guaranteed to exist between $c_1n^{1/2}$ and $c_2n^{3/4}$. 

Section 3 concerns generalisations of the following question from \cite{K}. 

\begin{q}\label{F:q}
What is the smallest family $\F\subseteq[n]^{(4)}$ with the property that every Hamilton cycle in $K_n$ contains a pair of edges whose union is an element of $\F$.
\end{q}

In the general question $\F$ is a family of $r$-sets and we require that our Hamilton cycle does not have a set of intervals of given lengths whose union is an element of $\F$.

To describe this generalisation more precisely, let $\F\subseteq[n]^{(r)}$, and $\x=(x_1,\dots,x_k)$ with $x_i$ positive integers and $\sum x_i=r$. We say that a Hamilton cycle in $G$ is $\x$-acceptable for $\F$ if we do not have intervals of vertices of lengths $x_1,\dots,x_k$ in the cycle whose union is an element of $\F$. We are interested in how small $|\F|$ can be if there is no $\x$-acceptable Hamilton cycle for $\F$.

Although this question is implicitly present in \cite{K}, no general results or conjectures are made there. We give upper and lower bounds for the smallest such $|\F|$, each of the form $c n^{r-k}$, and conjecture that our general upper bound is asymptotically tight. We prove results on a number of special cases including improved bounds when $\x=(2,2)$, an asymptotic determination of the extremal function when $\x=(2,\underbrace{1,\dots,1}_{r-2})$, and an answer to a question from \cite{K} on the $\x=(3,1)$ case.

In the final section we suggest a number of further questions and directions for study.

We note that a more general context for expressing this kind of problem is the notion of finding structures under constraints which forbid certain sets of edges all being used. We are given a graph $G$ and a family of sets of edges.
Under what conditions on $G$ and this family  can we guarantee the existence of some structure (typically a Hamilton cycle or a perfect matching) the edge set of which does not contain any member of the forbidden family. For example, the conditions of Theorem \ref{dks:thm} can be phrased as: for every pair of edges $(a,b),(c,d)\in E(G)$ for which $(b,c),(d,a)\in E(H)$ we cannot use both $(a,b)$ and $(c,d)$ in our Hamilton cycle. Both the constraints coming from a graph looked at in Section 2 and the constraints coming from a set system considered in Section 3 are also of this form. It may be interesting to consider different ways of expressing constraints of this kind in addition to the two considered here. 

Finally we remark that although the case that $G$ is a complete graph considered as \ref{F:q} looks rather special, we feel this may be a potentially good first step towards understanding the general behaviour. Another justification for this special case is that it naturally generalises some classical graph theoretic results. Dirac's theorem for instance can be stated (somewhat perversely) as: ``If $F\subseteq V^{(2)}$ is a set of forbidden edges then provided that no vertex is incident with more than $n/2$ elements of $F$ the graph $K_n$ contains a Hamilton cycle which does not contain any element of $F$.'' In our problem we replace the condition that certain edges are not allowed to be used with the condition that none of a certain family of \emph{sets} of edges is allowed to be all used .

\section{Matchings with Constraints given by a graph}

Throughout this section, by bipartite graph we mean a graph on vertex set $V=A\cup B$ with bipartition $(A,B)$ and $|A|=|B|=n$. In other words when we speak of several bipartite graphs they always have a single fixed bipartition. If $G$ and $H$ are bipartite graphs we say that a perfect matching $M$ in $G$ is \emph{acceptable} for $H$ if no two edges of $M$ form a 4-cycle with two edges of $H$. The general question is:

\begin{q}\label{match:q}
Under what conditions on the degrees of $G$ and $H$ can we guarantee the existence of a perfect matching in $G$ which is acceptable for $H$?
\end{q}

It is natural to consider the case that $G$ and $H$ are regular graphs. It will be convenient to assume further that the edge sets of $G$ and $H$ are disjoint. This brings us to Question \ref{regularmatch:q} from the introduction.

We have the following positive result which applies when $r$ is reasonably large.

\begin{theorem}\label{match-deg:thm}
Let $G$ be an $r$-regular bipartite graph and $H$ be an $s$-regular bipartite graph with $E(G)\cap E(H)=\emptyset$. Provided that $r>\frac{n}{2}+s^2$ there exists a perfect matching in $G$ which is acceptable for $H$.
\end{theorem}

\begin{proof}
Since $G$ is regular it certainly contains at least one perfect matching. Suppose, for a contradiction, that $G$ does not contain a perfect matching which is acceptable for $H$. Remove edges from $H$ one by one until one of the perfect matchings in $G$ becomes acceptable. Let $H'$ be the subgraph we are left with, and $e$ be the last edge removed. Let $M$ be any perfect matching in $G$ which is acceptable for $H'$. Label the vertices so that $A=\{a_1,\dots,a_n\}$, $B=\{b_1,\dots,b_n\}$ with $M=\{(a_i,b_i): 1\leq i\leq n\}$ and $e=(a_n,b_{n-1})$.

Adding $e$ to $H'$ makes $M$ unacceptable but this can only be because of the $4$-cycle $a_n,b_n,a_{n-1},b_{n-1}$. Hence, the matching $M\setminus (a_n,b_n)$ of $n-1$ edges is acceptable for $H'\cup e$. Let $Y=\Gamma_H(a_n)\subseteq B$ and $X=\Gamma_H(b_n)\subseteq A$. Let $X'=\{b_i: a_i\in X\}$, $Y'=\{a_i:b_i\in Y\}$. Finally, let $X''=\Gamma_H(X')\subseteq A$ and $Y''=\Gamma_H(Y')\subseteq B$. Since $H$ is $s$-regular $|X''|, |Y''|\leq s^2$. Since $deg_G(a_n)-|Y''|\geq r-s^2>n/2$ and similarly $deg_G(b_n)-|X''|\geq r-s^2>n/2$ we can find $t$ so that $(a_n,b_t),(a_t,b_n)\in E(G)$, $a_t\not\in X''$, $b_t\not\in Y''$. 

We claim that $M'=M\setminus \{(a_n,b_n),(a_t,b_t)\}\cup\{(a_t,b_n),(a_n,b_t)\}$ is acceptable for $H'\cup e$. This will contradict the definition of $H'$ and $e$. 

Suppose that we have a bad $4$-cycle containing two edges of $M'$ and two edges of $H'\cup e$. Since $M\setminus (a_n,b_n)$ is acceptable for $H'\cup e$ our bad cycle must contain at least one edge of $(a_t,b_n),(a_n,b_t)$. Further, since $a_n,b_n\not\in E(H)$ (as $E(G)\cap E(H)=\emptyset$) it must contain exactly one of the edges. Suppose that our bad cycle is $a_n,b_t,a_s,b_s$ with $(a_n,b_s),(a_s,b_t)\in E(H)$. This means that $b_s\in Y$, $a_s\in Y'$ and $b_t\in Y''$, contradicting our choice of $t$. Similarly, by the definition of $X,X'$ and $X''$, we cannot have that our bad cycle is $b_n,a_t,b_s,a_s$. It follows that $M'$ is acceptable for $H'\cup e$ and this contradiction completes the proof.
\end{proof}

A special case of this result gives a condition for an acceptable matching when $E(H)=E(K_{n,n})\setminus E(G)$. To express this we introduce some notation. Suppose that $G$ is a bipartite graph. We write $\overline{G}$ for the bipartite graph with edge set $\{ab: a\in A, b\in B, ab\not\in E(G)\}$. Let $b(n)$ be the minimum number $s$ such that there exists an $s$-regular bipartite graph $H$ for which there is no perfect matching in $\overline{H}$ which is acceptable for $H$.

\begin{cor}\label{match-comp-lb:cor}
For all $n$ we have that
\[
b(n)>\frac{1}{2}\left(\sqrt{2n+1}-1\right)=(\sqrt{2}+o(1))n^{1/2}.
\]
\end{cor}

\begin{proof}
If $H$ is an $s$-regular graph then $G=\overline{H}$ is an $n-s$ regular graph. Theorem \ref{match-deg:thm} shows that
we have a perfect matching in $G$ which is acceptable for $H$ if $n-s>\frac{n}{2}+s^2$. That is if $s<\frac{1}{2}\left(\sqrt{2n+1}-1\right)$.
\end{proof}

A construction gives an upper bound for $b(n)$. 

\begin{theorem}\label{match-comp-ub:thm}
For infinitely many $n$ we have that
\[
b(n)<(2+o(1))n^{3/4}.
\]
\end{theorem}

The construction we give works for all $n$ of the form $m\lfloor m^{1/3}\rfloor$ where $m=\frac{q^4-1}{q-1}$ with $q$ a prime power. It appears to be surprisingly difficult to deduce anything valid for all values of $n$ from this. The fact that we do not even know whether the function $b(n)$ is non-decreasing in $n$ contributes to this difficulty.

An ingredient of this construction is a bipartite graph with good expansion properties as described in the following result of Alon \cite{Alon}.

\begin{theorem}[Alon \cite{Alon}]\label{alon:thm}
For any integer $d\geq 1$ and $n=\frac{q^{d+1}-1}{q-1}$ with $q$ a prime power, there is an $r$-regular bipartite graph with $r=\frac{q^d-1}{q-1}=(1+o(1))n^{1-1/d}$ such that for all $0<x<n$ and any $X\subseteq A$ with $|X|=x$ we have $|\Gamma(X)|\geq n-\frac{n^{1+1/d}}{x}$.
\end{theorem}

\begin{proof}[Proof of Theorem \ref{match-comp-ub:thm}]
Let $m=\frac{q^4-1}{q-1}$ and $n=m\lfloor m^{1/3}\rfloor$. We need to construct a regular bipartite graph $H$ of degree $(2+o(1))n^{3/4}$ with the property that no perfect matching in $G=\overline{H}$ is acceptable for $H$. 

We first partition the vertices in $A$ into $k=\lfloor m^{1/3}\rfloor$ parts $A_1,A_2,\dots,A_k$ each of size $m$. Similarly we partition the vertices in $B$ into $k=\lfloor m^{1/3}\rfloor$ parts $B_1,B_2,\dots,B_k$ each of size $m$.

Let $B(m)$ be a bipartite graph with $m$ vertices in each part which satisfies the conditions of the $d=3$ case of Theorem \ref{alon:thm}. This graph is regular of degree $(1+o(1))m^{2/3}$.

We form $H$ by putting a copy of $K_{m,m}$ between each pair $A_i,B_i$ for $1\leq i\leq k$ and a copy of $B(m)$ between each pair $A_i,B_j$ with $i\not=j$. The graph $H$ is clearly regular of degree $m+(k-1)(1+o(1))m^{2/3}=(2+o(1))n^{3/4}$. 

Suppose that $M$ is any perfect matching in $G$. Let $Y=\{b\in B: (a,b)\in M \text{ for some } a\in A_1\}$ and $X=\{a\in B: (a,b)\in M \text{ for some } b\in B_1\}$. Since $|X|=|Y|=m$ and $X\cap A_1=Y\cap B_1=\emptyset$ there is some $s$ with $|X\cap A_s|\geq m/(k-1)>m^{2/3}$ and some $t$ with $|Y\cap B_t|\geq m/(k-1)>m^{2/3}$. Now, if $s=t$ we have a $4$-cycle consisting of an edge of $M$ between $A_1$ and $B_t$, an edge of $M$ between $A_s$ and $B_1$, and two edges from copies of $K_{m,m}$. If $s\not=t$ then the construction of $B(m)$ implies that there is some edge of $H$ between $X\cap A_s$ and $Y\cap B_t$. If not then $X\cap A_s$ has less than $m-|Y\cap B_t|\leq m-m^{2/3}$ neighbours in $B_s$ contradicting the defintion of $B(m)$. It follows that we have a $4$-cycle consisting of this edge, an edge of $M$ between $A_1$ and $B_t$, an edge of $M$ between $A_s$ and $B_1$, and an edge from the copy of $K_{m,m}$ between $A_1$ and $B_1$.

We conclude that $H$ does have the property that no perfect matching in $G$ is acceptable for $H$.
\end{proof}

In fact, if we only require an $H$ with small maximum degree with the property that no perfect matching in $\overline{H}$ is acceptable for $H$ a slightly simpler construction works. Assuming $n$ is of the form $\frac{q^4-1}{q-1}$ start with a copy of $B(n)$ as defined in the proof above. Take subsets $X\subseteq A$, $Y\subseteq B$ with $|X|=|Y|=n^{2/3}+1$ and add all edges between $X$ and $Y$ to form a graph $H$. Since there is an edge of $B(n)$ between any two subsets of size $n^{2/3}+1$, there is an edge of $B(n)$ between the vertices matched to $X$ and the vertices matched to $Y$ for any perfect matching in $\overline{H}$. It follows that no perfect matching in $\overline{H}$ is acceptable for $H$. The graph $H$ has maximum degree at most $2n^{2/3}+1$ although it is not regular.

Turning now to the more general question in which $G$ and $H$ are not required to be complementary Theorem \ref{match-deg:thm} is rather weak since the condition only holds when the degree of $G$ is quite large. For smaller degrees of $G$ we have a negative result obtained by taking copies of the graph constructed in the proof of Theorem \ref{match-comp-ub:thm}.

\begin{cor}
For any $r$ there exists, for infinitely many $n$, an $r$-regular graph $G$ and a $(2+o(1))r^{3/4}$-regular graph $H$ with disjoint edge sets and with the property that no perfect matching in $G$ is acceptable for $H$.
\end{cor}

\begin{proof}
Take vertex disjoint copies of the graphs constructed in the proof of Theorem \ref{match-comp-ub:thm}.
\end{proof}

Positive results in this small degree of $G$ case seem to be harder to prove. Indeed, we can not answer even the following apparently simple question:

\begin{q}\label{match-small:q}
Is there an integer $k$ such that for all sufficiently large $n$, if $G$ is a $k$-regular bipartite graph and $H$ is a 1-regular bipartite graph, we are guaranteed to have a perfect matching in $G$ which is acceptable for $H$?
\end{q}

We remark briefly that the conclusion does not hold if $k=2$ since such a 2-regular $G$ may have only 2 perfect matchings and it is easy to choose a 1-regular $H$ so that neither of them is acceptable (provided that $n\geq7$). However, it may be that even $k=3$ is sufficient for a positive answer.

\section{Constraints given by a set system}

\subsection{Hamilton cycles}

We turn now to constraints described in a different form.

If $\F\subseteq[n]^{(4)}$ and $G$ is a graph with vertex set $[n]$ we say that a Hamilton cycle in $G$ is a $(2,2)$-acceptable Hamilton cycle with respect to $\F$ if we do not have two disjoint edges of the cycle whose union is an element of $\F$. We will mainly be interested in the case when $G=K_n$. This leads us to Question \ref{F:q} from the introduction which can now be rephrased as:

\begin{q}\label{F22:q}
What is the smallest family $\F\subseteq[n]^{(4)}$ with the property that $K_n$ does not contain a $(2,2)$-acceptable Hamilton cycle with respect to $\F$?
\end{q}

A Hamilton cycle in $K_n$ can be thought of as a cyclic ordering of the vertices and we shall represent our Hamilton cycles by such a string of vertices. For a more general question, suppose that $G$ is a graph with vertex set $[n]$, $\F\subseteq[n]^{(r)}$, and $\x=(x_1,\dots,x_k)$ with $x_i$ positive integers and $\sum x_i=r$. We say that a Hamilton cycle $c_1,c_2,c_3,\dots,c_n$ in $G$ is $\x$-acceptable for $\F$ if we do not have $t_1,\dots,t_k$ with
\[
\bigcup_{i=1}^k\{c_{t_i+1},c_{t_i+2},\dots,c_{t_i+x_i}\}\in\F.
\]
(Here and elsewhere we interpret suffices modulo $n$.)

We generalise Question \ref{F22:q} in the obvious way.

\begin{q}\label{Fx:q}
Given $\x$ as above what is the smallest family $\F\subseteq[n]^{(r)}$ with the property that $K_n$ does not contain an $\x$-acceptable Hamilton cycle with respect to $\F$?
\end{q}

We will denote the size of this smallest family $\F$ by $m(\x,n)$.

This question of determining $m(\x,n)$ was essentially raised (with slightly different notation) in \cite{K} where a number of constructions relating to particular cases of it are given. Our aim is to improve the bounds in some of these and other special cases, and also to consider what general results may hold. We will generally be interested in asymptotic results in which $n$ tends to infinity with $\x$ (and hence $r$) being fixed.

A simple averaging argument gives a lower bound on $m(\x,n)$ for any $\x$.

\begin{theorem}\label{mxn-lb:thm}
Given $\x$ as above we have that
\[
m(\x,n)\geq\left(\frac{1}{c(\x)x_1!x_2!\dots x_k!}+o(1)\right)n^{r-k},
\]
where $c(\x)$ is the number of unordered partitions of an $r$-set into $k$ sets of sizes $x_1,x_2,\dots,x_k$ (and in particular does not depend on $n$).
\end{theorem}

Let $S$ be the set of all Hamilton cycles in $K_n$ (cyclic orderings of the vertices). Clearly $|S|=(n-1)!$. Given a set $F\in[n]^{(r)}$ we denote by $H(F)$ the set of all elements of $S$ which are not acceptable with respect to the single set $F$ (that is all cyclic orderings in which $F$ is a union of disjoint intervals of the appropriate size). Of course $H(F)$ also depends on $\x$ but it will always be clear from the context what $\x$ is and so this notation should cause no confusion.

\begin{proof}
If $\F$ is such that $K_n$ has no $\x$-acceptable Hamilton cycle with respect to $\F$ then $\bigcup_{F\in\F}H(F)=S$ and so $\sum_{F\in\F}|H(F)|\geq(n-1)!$. Now $|H(F)|$ does not depend on $F$ and $|H(F)|\leq c(\x)x_1!x_2!\dots x_k!(n-r+k-1)!$ (for certain $\x$ this will be an equality but in some cases it will be possible to have a Hamilton cycle for which $F$ can be represented as the union of suitable intervals in more than one way). It follows that
\[
c(\x)x_1!x_2!\dots x_k!(n-r+k-1)!|\F|\geq(n-1)!
\]
and so
\[
|\F|\geq\frac{(n-1)!}{c(\x)x_1!\dots x_k!(n-r+k-1)!}=\left(\frac{1}{c(\x)x_1!x_2!\dots x_k!}+o(1)\right)n^{r-k}
\]
as required.
\end{proof}

In fact this lower bound gives the correct order of magnitude for $m(\x,n)$ as the following theorem shows.

\begin{theorem}\label{mxn-ub:thm}
Given $\x$ as above and $n\geq r$ we have that
\[
m(\x,n)\leq\binom{n-k}{r-k}=\left(\frac{1}{(r-k)!}+o(1)\right)n^{r-k}.
\]
\end{theorem}

\begin{proof}
Let $\F=\{X\in[n]^{(r)}: \{1,2,\dots,k\}\subseteq X\}$. 

We will show that in any permutation of $[n]$ it is possible to find $k$ disjoint intervals of lengths $x_1,x_2,\dots,x_k$ such that
the union of these intervals contains $[k]$. This clearly shows that there is no $\x$-acceptable Hamilton cycle with respect to $\F$.

We will prove this claim by induction on $n$. If $n=1$ then the claim obviously holds. It is also clearly true if $n=r$. Suppose that $n>1$, $r<n$ and $c_1,\dots,c_n$ is our permutation. If $c_1\not\in[k]$ then applying the induction hypothesis to the permutation $c_2,\dots,c_n$ gives the result. If $c_1\in[k]$ then we will take our first interval to be $c_1,\dots,c_{x_1}$ and consider the permutation $c_{x_1+1},\dots,c_n$. Applying the induction hypothesis to this permutation with vector of interval lengths $(x_2,\dots,x_k)$ gives the result. (It may be that we have fewer than $k-1$ elements of $[k]$ contained in this permutation but this can only weaken the condition we need to satisfy.) 
\end{proof} 

If some of the $x_i$ are equal to 1 then the bound of Theorem \ref{mxn-ub:thm} is not sharp. In this situation we have the following stronger result.

\begin{theorem}\label{mxn-ub2:thm}
If $\x=(x_1,\dots,x_k)$ with $x_1\geq\dots\geq x_t>x_{t+1}=\dots=x_k=1$ then
\[
m(\x,n)\leq(1+o(1))\frac{\binom{n-t}{r-k}}{\binom{r-t}{r-k}}=\left(\frac{(k-t)!}{(r-t)!}+o(1)\right)n^{r-k}.
\]
\end{theorem}

\begin{proof}
Let $\mathcal{M}$ be the smallest family of $(r-t)$-subsets of $[n]\setminus[t]$ with the property that every $(r-k)$-subset of $[n]\setminus[t]$ is contained in at least one set in $\mathcal{M}$. By R\"odl's proof of the Erd\H os-Hanani conjecture \cite{Rodl} we have that $|\mathcal{M}|=(1+o(1))\frac{\binom{n-t}{r-k}}{\binom{r-t}{r-k}}=\left(\frac{(k-t)!}{(r-t)!}+o(1)\right)n^{r-k}$. Now let $\F=\{X\in[n]^{(r)}: [t]\subseteq X, X\setminus[t]\in \mathcal{M}\}$. We will show that there is no $\x$-acceptable Hamilton cycle with respect to $\F$. From which the required upper bound on $m(\x,n)$ follows.

As in the proof of Theorem \ref{mxn-ub:thm}, in any permutation of $[n]$ it is possible to find $t$ disjoint intervals $X_1,\dots,X_t$ of lengths $x_1,\dots,x_t$ such that the union of these intervals contains $[t]$. Now $\left(X_1\cup\dots\cup X_t\right)\setminus[t]$ is an $(r-k)$-subset of $[n]\setminus[t]$ and so there is some $(r-t)$-set $M\in\mathcal{M}$ which contains it. It follows that the intervals $X_1,\dots,X_t$ together with $k-t$ singleton intervals form a set in $\F$. It follows that no Hamilton cycle is $\x$-acceptable with respect to $\F$.
\end{proof}

We tentatively conjecture that the upper bounds of Theorems \ref{mxn-ub:thm} and \ref{mxn-ub2:thm} are asymptotically sharp under the appropriate conditions on $\x$.

\begin{conj}\label{mxn:conj}
If $x_i\not=1$ for all $i$ then  
\[
m(\x,n)=\left(\frac{1}{(r-k)!}+o(1)\right)n^{r-k}.
\]
If $\x=(x_1,\dots,x_k)$ with $x_1\geq\dots\geq x_t>x_{t+1}=\dots=x_k=1$ then
\[
m(\x,n)=\left(\frac{(k-t)!}{(r-t)!}+o(1)\right)n^{r-k}.
\]
\end{conj}

As we shall see later the exact upper bound of Theorem \ref{mxn-ub:thm} is not always correct even when $x_i\not=1$ for all $i$. This
slight improvement suggest that even if the conjecture is correct, the extremal families may have quite a complicated structure.

The next result improves the bound given by the averaging argument of Theorem \ref{mxn-lb:thm} (except in the trivial case when $x_i=1$ for all $i$).

\begin{theorem}\label{mxn-lb2:thm}
Given $\x$ as above with the $x_i$ not all equal to 1 we have that
\[
m(\x,n)\geq\frac{4kr+1}{4kr}\left(\frac{1}{c(\x)x_1!x_2!\dots x_k!}+o(1)\right)n^{r-k}.
\]
\end{theorem}

The proof is by showing that the assumptions made in the proof of Theorem \ref{mxn-lb:thm} cannot hold with equality. The main aim of this result is to demonstrate that Theorem \ref{mxn-lb:thm} is not sharp and we have not made a particular effort to obtain the strongest bound this method will give. We will however go through the details more carefully in one special case later. 

If $C$ is a Hamilton cycle (cyclic ordering) then let $d(C)$ be the number of $F\in\F$ for which $F$ consists of $k$ disjoint intervals of lengths $x_1,\dots,x_k$ in $S$ (in other words the number of $F\in\F$ for which $C\in H(F)$). Strictly $d(C)$ depends on $\F$ and $\x$ as well as $C$ but it will always be clear from the context what these are.

\begin{proof}
We will assume that $x_1=t$ is the maximum of the $x_i$. 

Let $\F$ be such that there is no $\x$-acceptable Hamilton cycle with respect to $\F$. We may assume also that $|\F|\leq\binom{n-k}{r-k}$ since we know that this is an upper bound for $m(\x,n)$. 

We have that $\sum_{F\in\F}|H(F)|=\sum_{C\in S}d(C)$. The property that there is no $\x$-acceptable Hamilton cycle with respect to $\F$ is equivalent to having $d(C)\geq1$ for all $C\in S$. Previously we used this bound on $d(C)$ to deduce a bound on $\F$. Here we will show that it is not possible for all the $d(C)$ to be this small and so obtain a stronger bound on their sum. Let 
\[
U'=\{C\in S : d(C)=1\}.
\]
For a cycle in $U'$ there is a unique $F\in\F$ for which $C\in H(F)$. Let $U$ be the set of all cycles in $U'$ for which the end points of any two of the $k$ intervals in $C$ whose union is this $F$ are at distance at least $r$ around $C$. We have that each $F\in \F$ gives rise to at most $c(n-r+k-2)!$ cycles in $U'\setminus U$ where $c$ is some constant depending only on $\x$. It follows that $|U'\setminus U|\leq |\F|c(n-r+k-2)!=o((n-1)!)$. Hence, if we can show that $|U|\leq\alpha(n-1)!$ then we will have $|U'|\leq(\alpha+o(1))(n-1)!$. Our aim is to prove such a bound by showing that each cycle in $U$ gives rise to a cycle not in $U'$.

Suppose that $C\in U$. Without loss of generality we may assume that $C=1,2,\dots,n$ and that $C\in H(F)$ where $F=\bigcup_{i=1}^k\{a_i+1,\dots,a_i+x_i\}$ with $a_1=1$. Consider the cycle $C'=1,2,\dots,t-1,t+1,t,t+2,\dots,n$ (that is $C$ with $t$ and $t+1$ exchanged). This is not in $H(F)$ since $F$ meets $C'$ in $k+1$ intervals (we are using here that $C\not\in U'$). Because there is no $\x$-acceptable Hamilton cycle we must have $C'\in H(F')$ for some $F'\in\F$. Now since $C\not\in H(F')$ we must have that either $t-1,t+1\in F'$, $t\not\in F'$ (Case 1) or $t,t+2\in F'$, $t+1\not\in F'$ (Case 2).
     
If we are in Case 1 then one of the $k$ intervals in $C'$ which is contained in $F'$ must be the length $t-i+1$ interval $i,i+1,\dots,t-1,t+1$. By our choice of $x_1$ to be the largest $x_i$ we have that $i\geq1$ and so the interval must be of this form. Let $C''$ be the cycle 
\[
1,2,\dots,i-1,t,i,i+1,\dots,t-1,t+1,t+2,\dots, n.
\]
Now it is clear that $C''\in H(F)$ and also $C''\in H(F')$ and so $d(C'')\geq2$.

If we are in Case 2 then suppose that one of the $k$ intervals in $C'$ which is contained in $F'$ is the length $i-t$ interval $t,t+2,t+3,\dots,i$. Let $C''$ be the cycle
\[
1,2,\dots,t-1,t,t+2,t+3,\dots,i,t+1,i+1,\dots,n.
\]
Now it is clear that $C''\in H(F')$. We also know, since the intervals of $C$ whose union is $F$ are at distance at least $r$ round the cycle, that none of $t+1,\dots,i$ are elements of $F$ (this is where we use the fact that we are working in $U$ and not just $U'$). It follows that $C''\in H(F)$ and so $d(C'')\geq2$.

Notice finally that in each case $C''$ was constructed from $C$ by moving one vertex by $x_i$ places in the cycle for some $i$. It follows that each cycle in $C''\in S\setminus U'$ can arise in this way from at most $4kr$ different cycles in $U$ (if $F,G\in\F$ with $C''\in H(F)\cap H(G)$ then the moved vertex must be in $F\cup G$ and there are at most $2k$ choices for where to move it). Hence
\[
|U|\leq 4kr|S\setminus U|
\]
and so
\[
|U|\leq\frac{4kr}{4kr+1}|S|
\]
Now since $U'\setminus U=o((n-1)!)$ we have that $U'\leq\left(\frac{4kr}{4kr+1}+o(1)\right)(n-1)!$. Using the same approach as in Theorem \ref{mxn-lb:thm}
\[
c(\x)x_1!x_2!\dots x_k!(n-r+k-1)!|\F|\geq(n-1)!+\left(\frac{1}{4kr+1}+o(1)\right)(n-1)!
\]
That is
\[
|F|\geq\left(\frac{4kr+1}{4kr}\frac{1}{c(\x)x_1!x_2!\dots x_k!}+o(1)\right)n^{r-k}
\]
\end{proof}

We now address some natural special cases of our problem of determining $m(\x,n)$. One, mentioned in \cite{K}, is the case
$\x=(2,2)$ which we referred to earlier. Theorems \ref{mxn-ub:thm} and \ref{mxn-lb:thm} show that
\[
\left(\frac{1}{12}+o(1)\right)n^2\leq m((2,2),n)\leq\left(\frac{1}{2}+o(1)\right)n^2.
\]
Going through a similar argument to Theorem \ref{mxn-lb2:thm} with more care and refining the construction of Theorem \ref{mxn-ub:thm} we are able to improve the constant in the lower bound and the $o(n^2)$ term in the upper bound as follows

\begin{theorem}\label{22:thm}
For some constant $c>0$ we have
\[
\left(\frac{1}{9}+o(1)\right)n^2\leq m((2,2),n)\leq\frac{1}{2}n^2-\frac{1}{2}n^{3/2}
\]
\end{theorem}

\begin{proof}
For the upper bound let $H$ be a graph with $V(H)=\{3,4,\dots,n\}$ which is $C_4$-free, triangle-free, and contains no Hamilton path. We will show that the there are no $(2,2)$-acceptable Hamilton cycles with respect to the family 
\[
\F=\{X\in[n]^{(4)}:1,2\in X, X\setminus\{1,2\}\not\in E(H)\}.
\]
The simplest case of the K\H ov\'ari-S\'os-Tur\'an Theorem \cite{KST} shows that there is a $C_4$-free bipartite graph with at most $1/2n^{3/2}+n$ edges. A suitable $H$ may be constructed from such a graph by deleting at most $n$ edges. It follows that we may take $|E(H)|=1/2n^{3/2}$ which will establish the upper bound.

To prove that $\F$ is as required we show that any Hamilton cycle has two disjoint edges whose union is an element of $\F$. Suppose that our Hamilton cycle is $C=c_1,\dots,c_n$ with $c_i=1$, $c_j=2$ and (without loss of generality) $1<i<j<n$. We consider three cases:
\begin{enumerate}
\item[Case 1:] If $j-i>2$ then consider the pairs $(c_{i-1},c_{j-1})(c_{i-1},c_{j+1}),(c_{i+1},c_{j-1}),(c_{i+1},c_{j+1})$. Since $H$ is $C_4$-free at least one of these pairs is not an edge of $H$. It follows that at least one of these pairs together with 1,2 forms a 4-set in $\F$ and so $C$ is not $(2,2)$-acceptable with respect to $\F$.
\item[Case 2:] If $j-i=2$ the consider the pairs $(c_{i-1},c_{i+1}),(c_{i-1},c_{i+3}),(c_{i+1},c_{i+3})$. Since $H$ is triangle-free at least one of these pairs is not an edge of $H$. It follows that at least one of these pairs together with 1,2 forms a 4-set in $\F$ and so $C$ is not $(2,2)$-acceptable with respect to $\F$.
\item[Case 3:] If $j-i=1$ then consider the pairs $(c_{j+1},c_{j+2}),(c_{j+2},c_{j+3}),\dots,(c_{i-2},c_{i-1})$. Since $H$ does not have a Hamilton path at least one of these pairs is not an edge of $H$. It follows that at least one of these pairs together with 1,2 forms a 4-set in $\F$ and so $C$ is not $(2,2)$-acceptable with respect to $\F$.
\end{enumerate}

For the lower bound let $\F$ be such that there is no $(2,2)$-acceptable Hamilton cycle with respect to $\F$ and assume that $|\F|\leq\binom{n-2}{2}$. As before we have that $\sum_{F\in\F}|H(F)|=\sum_{C\in S}d(C)$. Let 
\[
U'=\{C\in S : d(C)=1\}.
\]
For a cycle $C\in U'$ there is a unique $F\in\F$ for which $C\in H(F)$. Let $U$ be the set of all cycles in $U'$ for which the end points of the 2 edges whose union is this $F$ are at distance at least 3 around $C$. As in the proof of Theorem \ref{mxn-lb2:thm} $|U'\setminus U|=o((n-1)!)$. We also define
\begin{align*}
D&=\{C\in S : d(C)=2\}\\
T&=\{C\in S : d(C)\geq3\}.
\end{align*}

If $i$ and $j$ are consecutive elements of a cyclic ordering $C$ we will denote by $\pi_{i,j}(C)$ the cyclic ordering formed by swapping $i$ and $j$.We will refer to cycles $B$ and $C$ with $C=\pi_{i,j}(B)$ as being neighbouring.

Suppose that the cycle $C$ has $d(C)=1$ and let $F=\{a,b,x,y\}\in\F$ with $a$ and $b$ consecutive in $C$ and $x$ and $y$ consecutive in $C$. Suppose that $a,b,c,d$ are consecutive in $C$ and consider the cycle $C'=\pi_{b,c}(C)$. There is a set $F'\in\F$ with $C'\in H(F')$ and we must have either $a,c\in F'$ (Case 1) or $b,d\in F'$ (Case 2). Since $C\not\in H(F')$ we cannot have that $F'=\{a,b,c,d\}$ and so exactly one of these cases occurs.

If we are in Case 1 then the cycle $C''=\pi_{a,b}(C)$ is in both $H(F)$ and $H(F')$ and so $d(C'')\geq2$.

If we are in Case 2 then the cycle $C''=\pi_{d,c}(C)$ is in both $H(F)$ and $H(F')$ and so $d(C'')\geq2$ (note that here we need the condition which distinguishes $U$ from $U'$ to ensure that $C''\in H(F)$).

Repeating the same argument starting from $C'=\pi_{z,a}(C)$ where $z$ is the predecessor of $a$ in $C$ we obtain that $C$ has either
at least 2 neighbouring cycles in $D\cup T$ or at least 1 neighbouring cycle in $T$.

Repeating the same argument starting with the pair $x,y$ we obtain that $C$ has one of the following:
\begin{itemize}
\item at least 2 neighbouring cycles in $T$,
\item at least 3 neighbouring cycle in $D\cup T$ of which at least 1 is in $T$,
\item at least 4 neighbouring cycles in $D\cup T$.
\end{itemize}
Writing $e(U,T)$ for the number of pairs of neighbouring cycles with one in $U$ and one in $T$ and similarly for $e(U,D)$ we obtain that
\[
4|U|\leq 2e(U,T)+e(U,D).
\]
It is easy to see that a cycle in $D\cup T$ can neighbour at most 8 cycles in $U$ and so $e(U,T)\leq 8|T|$, $e(U,D)\leq8|D|$. From this we get that
\[
|U|\leq 4|T|+2|D|.
\]
Now 
\[
12(n-3)!|\F|=\sum_{F\in\F}|H(F)|=\sum_{C\in S}d(C)\geq|S|+|D|+2|T|\geq|S|+\frac{1}{2}|U|.
\]
Also
\[
12(n-3)!|\F|\geq|S|+(|S|-|U'|)=2|S|-|U|+o((n-1)!).
\]
Taking whichever of these bounds is stronger depending on $|U|$ we conclude that
\[
12(n-3)!|\F|\geq\frac{4}{3}|S|+o((n-1)!)
\]
and so 
\[
|\F|\geq\left(\frac{1}{9}+o(1)\right)n^2.
\]
\end{proof}

Another natural case is $\x=(2,1,..,1)$ . Here we have a lower bound which agrees asymptotically with the upper bound from Theorem \ref{mxn-ub2:thm}. 

\begin{theorem}\label{211:thm}
Let ${\bf x}=(2,\underbrace{1,\dots,1}_{r-2})$ with $r$ fixed. We have 
\[
m(\x,n)=\left(\frac{1}{r-1}+o(1)\right)n.
\]
\end{theorem}

It is worth noting that in this case $m(\x,n)$ can be expressed more directly in graph theoretic terms; it is the smallest number of copies of $K_r$ whose removal makes $K_n$ non-Hamiltonian. We will use this formulation in the proof below. We will also the use the well known Bondy-Chv\'atal Theorem \cite{BC} which characterises Hamiltonian graphs.

\begin{theorem}[Bondy-Chv\'atal]
Let $G$ be a graph on $n$ vertices and $x,y$ be two non-adjacent vertices in $G$ with $\deg(x)+\deg(y)\geq n$. Then $G$ is Hamiltonian if and only if the graph formed by adding the edge $xy$ to $G$ is Hamiltonian.
\end{theorem}

\begin{proof}[Proof of Theorem \ref{211:thm}]
The upper bound follows from Theorem \ref{mxn-ub2:thm}. In fact it is also rather easy to describe the construction directly. Let $P$ be a family of $(r-1)$-subsets of $\{2,\dots,n\}$ with $|P|=\lceil\frac{n-1}{r-1}\rceil$ whose union is $\{2,\dots,n\}$ (if $r-1$ divides $n-1$ we simply take a partition if not then we keep the overlap as small as possible). Now let $\F=\{\{1\}\cup X: X\in P\}$. If $C=c_1,\dots,c_n$ is a Hamilton cycle with $c_1=1$ then $c_2\in X$ for some $X\in P$ and then $\{1\}\cup X$ contains two consecutive elements of $C$. This means that $\F$ has no $\x$-acceptable Hamilton cycle and so $m(\x,n)\leq\lceil\frac{n-1}{r-1}\rceil$.

For the corresponding lower bound we use the fact that $m(\x,n)$ is equal to the smallest number $t$ for which we can delete $t$ copies of $K_r$ from $K_n$ and be left with a non-Hamiltonian graph. Suppose that the graph we are left with after deleting these copies of $K_r$ is $G$ and let $\deg_G(i)=d_i$. We will say that a pair $x,y\in[n]$ with $xy\not\in E(G)$ is \emph{bad} if $d_x+d_y<n$. By the Bondy-Chv\'atal Theorem $G$ must remain non-Hamiltonian even after we add to it all pairs which are not bad. So we must, by Dirac's theorem, have a vertex $x$ which is incident with at least $n/2$ bad edges. 

Let $f=n-1-d_x$ be the number of edges in $K_n$ incident with $x$ which are deleted (that is those that are contained in some deleted $K_r$). Each $K_r$ removed from $K_n$ contributes at most $r-1$ to $f$ and so 
\[
t\geq\frac{f}{r-1}.
\] 
For the pair $x,y$ to be bad we must have $d_y\leq f$ and so at least $n-1-f$ edges incident with $y$ must be deleted. In total we have $n/2$ 
vertices $y$ for which this holds and each edge deleted is incident with at most 2 of them. It follows that at least $\frac{n(n-1-f)}{4}$ edges must be deleted. Since each $K_r$ contributes $\binom{r}{2}$ to the total number of deleted edges, we have that
\[
t\geq\frac{n(n-1-f)}{2r(r-1)}.
\]
The first bound is increasing with $f$, the second bound is decreasing with $f$. They are equal when $f=\frac{n(n-1)}{n+2r}$ at which point their common value is $\frac{n(n-1)}{(n+2r)(r-1)}$. It follows that
\[
t\geq\frac{n(n-1)}{(n+2r)(r-1)}=(1-o(1))\frac{n}{r-1}.
\]
\end{proof}

A further instance of the general problem which we mention briefly is the case $\x=(r-1,1)$. The upper bound we get here from Theorem \ref{mxn-ub2:thm} is
\[
m((r-1,1),n)\leq\left(\frac{1}{(r-1)!}+o(1)\right)n^{r-2}.
\]
In particular $m((3,1),n)\leq(\frac{1}{6}+o(1))n^2$. This answers in the negative a question of Katona (Problem 5 from \cite{K}) which essentially asked ``Is it true that if $|\F|\leq(\frac{1}{4}+o(1))n^2$ then $K_n$ contains a Hamilton cycle which is both $(3,1)$-acceptable and $(2,2)$-acceptable for $\F$?''

\subsection{Other graphs}

It is possible to raise similar questions in which the structure we are interested in is something other than a Hamilton cycle.
Probably the most natural sort of structure is a spanning subgraph of some simple form. One simple variant is to consider Hamilton paths.

\begin{q}\label{F22-path:q}
What is the smallest family $\F\subseteq[n]^{(4)}$ with the property that $K_n$ does not contain a $(2,2)$-acceptable Hamilton path with respect to $\F$?
\end{q}
Where, naturally, a Hamilton path is $(2,2)$-acceptable if we do not have two disjoint edges of the path whose union is an element of $\F$.

We denote the size of this smallest family $\F$ by $p((2,2),n)$. The main reason for raising this question is that there are indications (see below) that the extremal families for this variant may have a simpler structure. This suggests that it may be good test case for developing approaches to this type of problem.

\begin{theorem}\label{22hpath:thm}
\[
\frac{1}{6}(n-3)(n-2)\leq p((2,2),n)\leq\binom{n-2}{2}
\]
\end{theorem}

\begin{proof}
For the upper bound note that if $\F=\{X\in[n]^{(4)}:1,2\in X\}$ there is no $(2,2)$-acceptable Hamilton path with respect to $\F$.

For the lower bound consider, as in the proof of Theorem \ref{mxn-lb:thm}, the set $S$ of all Hamilton cycles in $K_n$ and denote by $H(F)$ the set of all elements of $S$ which are not acceptable with respect to the single set $F$. If $\F$ is such that $K_n$ has no $\x$-acceptable Hamilton path with respect to $\F$ then every element of $S$ is contained in at least $2$ of the $H(F)$ with $F\in\F$ (if a Hamilton cycle was contained in a unique $H(F)$ then deleting one of the edges whose union is $F$ would give an acceptable Hamilton path). It follows that 
$\sum_{F\in\F}|H(F)|\geq2(n-1)!$ and since $|H(F)|=12(n-3)!$ the bound follows.
\end{proof}

The upper bound construction is the same as that given for the Hamilton cycle case in Theorem \ref{mxn-ub:thm}. However, in contrast
to the Hamilton cycle problem this construction is minimal in that we cannnot remove any $4$-set from it without making some Hamilton path $(2,2)$-acceptable. It is possible that this upper bound is exactly sharp and if this is the case then the simple structure of the extremal family may make the problem easier that the Hamilton cycle case. 

\section{Further Questions}

Several questions and conjectures have been mentioned in earlier sections. In this section we summarise these and collect a few other possible questions and directions for further study, concentrating mainly on the situation where our constraints are given by a set system.

For matchings under constraints given by a graph the most obvious question is to bound $b(n)$ more tightly. We have no feeling for where the true order of magnitude should lie between $n^{1/2}$ and $n^{3/4}$. In addition it would be nice to know more about the general behaviour of $b$; for instance is it a non-decreasing function. The $k=3$ case of Question \ref{match-small:q} is also an appealing problem which we repeat here.

\begin{q}\label{match-3:q}
Is it true that for all sufficiently large $n$, if $G$ is a $3$-regular bipartite graph and $H$ is a 1-regular bipartite graph, we are guaranteed to have a perfect matching in $G$ which is acceptable for $H$?
\end{q}

For Hamilton cycles under constraints given by a set system the main open problem is Conjecture \ref{mxn:conj} on the asymptotic behaviour of $m(\x,n)$. As well as this asymptotic behaviour we could ask for exact values of $m(\x,n)$. This is probably a much harder problem even for our main example $\x=(2,2)$ (at least if Conjecture \ref{mxn:conj} is correct). However, as we indicated earlier, there may be variants of the problem for which this is a more approachable question. For instance:

\begin{q}
Is it true that if $\F\subseteq[n]^{(4)}$ with $|\F|<\binom{n-2}{2}$ then $K_n$ does not contain a $(2,2)$-acceptable Hamilton path with respect to $\F$?
\end{q}

It may be worth seeking out other variants along these lines for which the conjectured extremal family has a simple structure.

Another direction also raised in \cite{K} is to consider degree versions. Given $0\leq t<r$ and $F\subseteq[n]^{(r)}$ we let $d_t(\F)$ be the maximum $t$-degree of $\F$; that is the maximum over all $t$-subsets $D\in[n]^{(t)}$ of the number of elements of $\F$ which contain $D$. We now define $m_t(\x,n)$ to be the smallest $t$-degree $d_t(\F)$ where $\F\subseteq[n]^{(r)}$ is a family with the property that $K_n$ does not contain an $\x$-acceptable Hamilton cycle with respect to $\F$. This is a generalisation of our earlier question since $d_0(\F)=|\F|$ and so $m(\x,n)=m_0(\x,n)$.

For the family $\F=\{\{1,x,x+1,y\}: 2\leq x\leq n-1, y\not=1,x,x+1\}\cup\{\{1,2,y,n\}: y\not=1,2,n\}$ there are no $(2,2)$-acceptable Hamilton cycles in $K_n$. Also no pair is contained in more that $3n-13$ of the $4$-sets in $\F$. It follows that $m_2((2,2),n)\leq 3n-13$. It seems plausible that $m_2((2,2),n)$ is linear in $n$ but we could not prove this. For $m_1(\x,n)$ it is trivial that $\frac{m_0(\x,n)}{n}\leq m_1(\x,n)\leq m_0(\x,n)$ so the order of magnitude of $m_1(\x,n)$ is between $n^{r-k-1}$ and $n^{r-k}$.

In all our constructions $\F$ is very asymmetric. This leads us to the following question: 

\begin{q}
What is the smallest vertex-transitive $\F$ with the property that $K_n$ does not contain a $(2,2)$-acceptable Hamilton cycle with respect to $\F$? In particular is there such a family with $|\F|=cn^2$ for some $c$?
\end{q} 

We suspect that the answer to the second of these questions is no. The connection between Hamilton cycles and symmetry is tantalising (consider for instance Lov\'asz's famous question of whether all but finitely many connected vertex-transitive graphs are Hamiltonian \cite{Lov}) and questions like the one above may give some insight into it. There is also a connection between this problem and the degree variants mentioned above since a vertex-transitive $\F$ would have $d_1(\F)=\frac{|\F|}{n}$ and so if the above question has a positive answer then $m_1((2,2),n)$ is linear in $n$.

Recalling Katona's question on the minimum size of an $\F$ with no Hamilton cycle which is both $(2,2)$-acceptable and $(3,1)$-acceptable it would be possible to pose the more general problem involving more than one $\x$. Specifically, if $\x=(x_1,\dots,x_k)$, $\y=(y_1,\dots,y_l)$ with $x_i$ positive integers and $\sum x_i=\sum y_i=r$ then what is the smallest $\F$ for which $K_n$ has no Hamilton cycle which is both $\x$-acceptable and $\y$-acceptable. The size of such an $\F$ is necessarily smaller than both $m(\x,n)$ and $m(\y,n)$. It would be interesting to know whether behaviour analogous to principality of the graph Tur\'an density exists. That is, whether the size of the minimal $\F$ is essentially equal to whichever of of $m(\x,n)$ and $m(\y,n)$ is smallest or whether it can be significantly smaller than both of them.

Finally, could these results on the existence of $\x$-acceptable Hamilton cycle in $K_n$ be extended to $\x$-acceptable Hamilton cycles in an arbitrary $r$-regular graph $G$?

\section{Acknowledgement}
Part of this work was carried out during the author's visit to Budapest in March 2011. I am grateful to Gyula Katona for 
several useful and interesting conversations. I also thank the London Mathematical Society and the Alfr\'ed R\'enyi Institute for providing funding for this visit.

\end{document}